\documentclass[12pt, reqno]{amsart}
\usepackage[bookmarksnumbered,colorlinks, plainpages]{hyperref}
\hypersetup{colorlinks=true,linkcolor=red, anchorcolor=green, citecolor=cyan, urlcolor=red, filecolor=magenta, pdftoolbar=true}
\usepackage{fixltx2e}
\usepackage[latin1]{inputenc}
\RequirePackage{tkz-berge}
\usepackage{amsmath, amsthm, amscd, amsfonts, amssymb, graphicx, color, enumerate, xcolor,  mathrsfs, latexsym}

 \definecolor{cupgreen}{rgb}{0,0.498,0.208}
  \definecolor{cupblue}{rgb}{0,0,.5}
  \definecolor{cupred}{rgb}{1,0.04,0}
  \definecolor{cuppink}{rgb}{0.925,0,0.545}
  \definecolor{cupmagenta}{rgb}{0.624,0.161,0.424}
  \definecolor{cupbrown}{rgb}{0.71,0.212,0.133}

  \definecolor{cupgreen}{rgb}{0,0,0}
  \definecolor{cupblue}{rgb}{0,0,0}
  \definecolor{cupred}{rgb}{0,0,0}
  \definecolor{cuppink}{rgb}{0,0,0}
  \definecolor{cupmagenta}{rgb}{0,0,0}
  \definecolor{cupbrown}{rgb}{0,0,0}

\definecolor{TITLE}{rgb}{0,0,0}
\definecolor{AUTHOR1}{rgb}{0.00,0.59,0.00}
\definecolor{AUTHOR2}{rgb}{0.50,0.00,1.00}
\definecolor{SECTION}{rgb}{0.50,0.00,1.00}
\definecolor{FOOTTITLE}{rgb}{0.00,0.50,0.75}
\definecolor{THM}{rgb}{0.8,0,0.1}
\definecolor{SEC}{rgb}{0,0,1}

\newcommand{\N}{\mathcal N}

\makeatletter
\normalsize
\newtheorem{theorem}{{\color{THM} Theorem}}[section]
\newtheorem{cor}[theorem]{{\color{THM}Corollary}}
\theoremstyle{definition}

\newtheorem{Def}[theorem]{{\color{THM}Definition\ }}

\newtheorem{exa}[theorem]{{\color{THM}Example}}

\newtheorem{remark}[theorem]{{\color{THM}Remark}}
\numberwithin{equation}{section}

\def\stirlingfirst#1#2{\genfrac{[}{]}{0pt}{}{#1}{#2}}
\def\stirling2#1#2{\genfrac{\{}{\}}{0pt}{}{#1}{#2}}
\def\lah#1#2{\genfrac{\lfloor}{\rfloor}{0pt}{}{#1}{#2}}
\usepackage{amscd}
\begin{document}
\title{Mixed coloured permutations}

\author{Be\'{a}ta B\'{e}nyi}
\address{Faculty of Water Sciences, National University of Public Service, Hungary}
\email{beata.benyi@gmail.com}

\author{Daniel Yaqubi}
\address{Faculty of Agriculture and Animal Science, University of Torbat-e Jam, Iran .}
\email{daniel\_yaqubi@yahoo.es}
.
\begin{abstract}
In this paper we introduce mixed coloured permutation, permutations with certain coloured cycles, and study the enumerative properties of these combinatorial objects. We derive the generating function, closed forms, recursions and combinatorial identities for the counting sequence, mixed Stirling numbers of the first kind. In this comprehensive study we consider further the conditions on the length of the cycles, $r$-mixed Stirling numbers and the connection to Bell polynomials.   
\end{abstract}

\maketitle
\section{Introduction}
Permutations are one of the richest combinatorial objects. A permutation of a set $S$ with $n$ elements can be viewed as a function $w:\{1,2,\ldots, n\}\rightarrow S$; $w(i)=w_i$ or as a bijection $w:S\rightarrow S$. In this latter case for each element $x\in S$ there is a unique $\ell $ such that $x,w(x),\ldots, w^{\ell-1}(x)$ are different and $w^{\ell}(x)=x$. The sequence $(x,w(x,)\ldots w^{\ell-1}(x))$ is called a \emph{cycle} of length $l$. Any permutation $w$ is a product of distinct cycles, and this decomposition is unique. It is well-known that the number of permutations of $[n]=\{1,2,\ldots,n\}$ with exactly $k$ cycles is the signless Stirling number of the first kind, $\stirlingfirst{n}{k}$.  Clearly, summing up for $k$ we obtain all permutations with a total number of $n!$. 
We introduce \emph{coloured permutation} by colouring the cycles. 
\begin{Def}
A permutation $\pi$ with cycle decomposition $\{C_1,C_2,\ldots C_m\}$ and an assignment of a colour from the set $\{1,2,\ldots, k\}$, such that $t_i$ cycles obtain the colour $i$ is called a \emph{$(t_1,t_2,\ldots, t_k)$-coloured permutation}. 
\end{Def}
We denote the number of $(t_1,t_2,\ldots, t_k)$-coloured permutations of $[n]$ by 
\[\stirlingfirst{n}{t_1,t_2,\ldots,t_k}.\] 
\begin{exa}
The following example is a $(3,1,1)$-coloured permutation of $[9]$.
\begin{align*}
\textcolor{red}{(1\, 6\,10)}\textcolor{blue}{(2\,4)}\textcolor{green}{(3\,11\,9)}\textcolor{red}{(5\,7)}\textcolor{red}{(8)}\quad \Longleftrightarrow\quad
\textcolor{red}{6}\,\textcolor{blue}{4}\,\textcolor{green}{11}\,\textcolor{blue}{2}\,\textcolor{red}{7\,10\,5\,8}\,\textcolor{green}{3}\,\textcolor{red}{1}\,\textcolor{green}{9}
\end{align*}
\end{exa}
A permutation with $k$ cycles can be coloured by $k$ distinct colours in $k!$ ways; hence, the number of \emph{distinctly coloured permutations} is  \[\stirlingfirst{n}{1,1,\ldots,1}=k!\stirlingfirst{n}{k}.\] 
Though in this paper we focus on permutations, we can formulate the problem in a more general form, considering permutations of multisets instead of sets. 
Let $\mathcal{B}=(1^{b_1},2^{b_2},\ldots, n^{b_n})$ and $\mathcal{C}=(1^{c_1},2^{c_2},\ldots k^{c_k})$ be two multisets ($b_i$ (resp.~ $c_i$) denotes the appearence of the element $i$ in the set). We let $\stirlingfirst{\mathcal{B}}{\mathcal{C}}$ denote the number of permutations of the $b_1+b_2+\cdots +b_n$ elements of the multiset $\mathcal{B}$ into exactly $c_1+c_2+\cdots +c_k$ cycles, such that $c_j$ cycles are labeled by $j$. Further, let $\mathcal{J}=\{1^{j_1},2^{j_2},\ldots, k^{j_k}\}$. Then 
\[\stirlingfirst{\mathcal{B}}{\mathcal{C}}_0=\sum_{0\leq i\leq k,0\leq j_i\leq c_i}\stirlingfirst{\mathcal{B}}{\mathcal{J}}.\]
Clearly, for $b_1=b_2=\cdots=b_n=1$ and $c_1=m$, $c_2=c_3=\cdots=c_k=0$, we have 
\begin{align*}
\stirlingfirst{\mathcal{B}}{\mathcal{C}}=\stirlingfirst{n}{m}.
\end{align*}
As we mentioned before for $b_1=b_2=\cdots=b_n=1$ and $c_1=c_2=c_3=\cdots=c_k=1$, we have $\stirlingfirst{\mathcal{B}}{\mathcal{C}}=k!\stirlingfirst{n}{k}$. In this case the counting sequence $\stirlingfirst{\mathcal{B}}{\mathcal{C}}_0$ is referred to in OEIS as A006252 \cite{OEIS}.
Now, we give a formula for the special case with $b_1=b_2=\cdots=b_n=1$, but arbitrary $\mathcal{C}$. This special case  $b_1=b_2=\cdots=b_n=1$ corresponds to permutations, and we put simple $n$ instead of $\mathcal{B}$.
\begin{theorem}
Let $n$ be a positive integer, and $t_1,\ldots,t_k\in\mathbb{N}$. The number of $(t_1,\ldots, t_k)$-coloured permutations is given by the following formula. 
\begin{align}\label{genform}
\stirlingfirst{n}{t_1,t_2,\ldots, t_k}&=\sum_{\substack {\ell_1+\cdots+\ell_k=n}}\binom{n}{\ell_1,\ell_2,\ldots, \ell_k}\stirlingfirst{\ell_1}{t_1}\stirlingfirst{\ell_2}{t_2}\cdots\stirlingfirst{\ell_k}{t_k}.
\end{align}
where $\binom{n}{\ell_1,\ldots,\ell_k}$ is the multinomial coefficient defined by $\frac{n!}{\ell_1!\ldots \ell_k!}$ with $\ell_1+\cdots +\ell_k=n$.
\end{theorem}
\begin{proof}
We constitute $(t_1,t_2,\ldots, t_k)$-coloured pemutations of the set $[n]$ as follows. First, choose $\ell_1$ elements in ${n \choose \ell_1}$ ways  (label them $1$) and order into $t_1$ cycles in $\stirlingfirst{\ell_1}{t_1}$ ways. Next, choose $\ell_2$ out of the remaining $n-\ell_1$ elements in ${n-\ell_1 \choose \ell_2}$ ways and order these elements into $t_2$ cycles in $\stirlingfirst{\ell_1}{t_2}$ ways. By continuing the process we obtain the theorem. 
\end{proof}
\begin{theorem}
The number of $(t_1,t_2,\ldots, t_k)$-coloured permutations is given by
\[\stirlingfirst{n}{t_1,t_2,\ldots, t_k}=\sum_{1\leqslant i\leqslant k,0\leqslant j_i\leqslant t_i} (-1)^{\sharp(j_1,\ldots,j_k)}\stirlingfirst{n}{j_1,j_2,\ldots, j_k}_0,\]
where $\sharp(j_1,\ldots,j_k)$ is the number of $i$'s such that $j_i\neq0$. 
\end{theorem}
\begin{proof}
The theorem follows by the Inclusion-Exclusion principle from the definitions.
\end{proof}
\begin{theorem}
For the number of $(t_1,t_2,\ldots, t_k)$-coloured permutations the following recurrence holds
\begin{align*}\stirlingfirst{n}{t_1,t_2,\ldots, t_k}=(n-1)\stirlingfirst{n-1}{t_1,t_2,\ldots, t_k}+\sum_{j=1}^k\stirlingfirst{n-1}{t_1,\ldots, t_{j-1},t_j-1,t_{j+1}\ldots, t_k}.
\end{align*}
\end{theorem}
\begin{proof}
Consider the $n$th element. If it is a singleton coloured by the colour $j$, the remaining elements construct a $(t_1,\ldots, t_{j-1}, t_j-1,t_{j+1},\ldots, t_k)$-coloured permutation. If it is not a singleton, we can insert it before any of the elements, and join it to the cycle of this element, which gives $(n-1)\stirlingfirst{n-1}{t_1,t_2,\ldots, t_k}$ possibilities. 
\end{proof}
In the rest of the paper we consider a special case.
\begin{Def} We call a $(t,1,\ldots, 1)$-coloured permutations \emph{mixed coloured permutations}. We denote the set of mixed coloured permutations by $\mathcal{MC}(n,k,t)$ and denote the size of this set by $\stirlingfirst{n}{t,1,\ldots,1}=\stirlingfirst{n}{k/t}$. We call the number sequence $\stirlingfirst{n}{k/t}$ \emph{mixed Stirling number of the first kind}. We refer to the colour, that are used for $t$ cycles, as the \emph{special colour}.
\end{Def}
Table 1. lists the number of mixed coloured permutations for some small values of $n$, $k$, and $t$.
\begin{table}[h]
\begin{tabular}{c|ccccc}
$n/k$&1&2&3&4&5\\\hline
2&1&&&&\\
3&3&3&&&\\
4&11&18&12&&\\
5&50&105&120&60&\\
6&274&675&1020&900&360
\end{tabular}\quad\quad
\begin{tabular}{c|ccccc}
$n/k$&1&2&3&4&5\\\hline
3&1&&&&\\
4&6&4&&&\\
5&35&40&20&&\\
6&225&340&300&120&\\
7&1624&2940&3500&2520&840
\end{tabular}
\caption{$\stirlingfirst{n}{k/2}$ and $\stirlingfirst{n}{k/3}$}
\end{table}
For the special case $t=1$, when every cycle is distinctly coloured, we have the following recurrence relation: 
\begin{theorem}For positive integers $n$, $k$, with $k\leq n$ $\stirlingfirst{n}{k/1}$ satisfies 
\begin{align}\label{OCrec}
\stirlingfirst{n}{k/1} =k\stirlingfirst{n-1}{k-1/1}+(n-1)\stirlingfirst{n-1}{k/1}.
\end{align}
\end{theorem}
\begin{proof}
Consider the $n$th element. If it is a fixed point, then there are $\stirlingfirst{n-1}{k-1/1}$ ways to create the mixed coloured permutation from the remaining elements. The cycle with $n$, can be coloured by any of the $k$ colours. If it is not a fixed point, proceed as follows: write the permutation in cycle notation (the elements in a cycle are arranged so that the least element is written first). This can be done in $\stirlingfirst{n-1}{k/1}$ ways. Insert now the element $n$ before any of the elements. If we insert $n$ before an element which started a cycle, $n$ will be included into this cycle. This gives $(n-1)\stirlingfirst{n-1}{k/1}$ possibilities. 
\end{proof}
\begin{theorem}\label{BBB}
Let $n$,$k$ and $t$ be positive integers with $t,k\leq n$. Then the number of mixed coloured permutations is
\[\stirlingfirst{n}{k/t}=\sum_{\ell=t}^{n-k+1}(k-1)!{n\choose \ell}\stirlingfirst{\ell}{t}\stirlingfirst{n-\ell}{k-1/1}.\]
\end{theorem}
\begin{proof}
Choose $\ell\geqslant t$ elements in $n\choose \ell$ ways and order them into $t$ cycles. These cycles are coloured by the special colour. The remaining $n-\ell$ elements have to be ordered into $k-1$ distinctly coloured cycles. This can be done in $\stirlingfirst{n-\ell}{k-1/1}$ ways. Note that we should have $k-1\leqslant n-\ell$.
\end{proof}
\begin{cor}
Let $n,k$ and $t$ be positive integers with $k,t\leq n$. Then we have
\begin{itemize}
\item[i.]$\stirlingfirst{n}{k/0}=\stirlingfirst{n}{k-1/1}$
\item[ii.]$\stirlingfirst{n}{1/t}=\stirlingfirst{n}{t}$,
\item[iii.]$\stirlingfirst{n}{1/n-1}=\binom{n}{2}=\stirlingfirst{n}{n-1}$,
\item[iv.]$\stirlingfirst{n}{2/n-1}=n$,
\item[v.] $\stirlingfirst{n}{n-t+1/t}=\frac{n!}{t!}$.
\end{itemize}
\end{cor}
We derive now several formulas using combinatorial arguments. (We use the notation of falling factorial $(n)_k=n(n-1)\cdots(n-k+1)$).
\begin{theorem}
For positive integers $n,k,t$ with $k,t\leq n$, the number of mixed coloured permutations is 
\begin{align}\label{closed1}
\stirlingfirst{n}{k/t}=(t+k-1)_{k-1}\stirlingfirst{n}{t+k-1},
\end{align}
\end{theorem}
\begin{proof}First, we arrange the $n$ elements into $t+k-1$ cycles, then we choose $t$ cycles to colour them with the special colour and colour the remaining $k-1$ cycles with distinct colours.
\end{proof}
\begin{theorem}
For positive integers $n,k,t$ with $k,t\leq n$, we have
\begin{align}\label{closed2}
\stirlingfirst{n}{k/t}=\sum_{j=t}^{n-k+1}(k-1)!\binom{n}{j}\stirlingfirst{j}{t}\stirlingfirst{n-j}{k-1}.
\end{align}
\end{theorem}
\begin{proof}
We choose first $j$ elements and create from these elements $t$ in $\binom{n}{j}\stirlingfirst{j}{t}$ ways. The remaining $n-j$ elements are included in the $k-1$ cycles, and since these are distinctly coloured, we have a factor $(k-1)!\stirlingfirst{n-j}{k-1}$. 
\end{proof}
The first recursion is based on the position of the $n$th element. 
\begin{theorem}
For positive integers $n,k,t$ with $k,t\leq n$ the following holds
\begin{align}\label{recur1}
\stirlingfirst{n}{k/t}=\sum_{j=1}^{n-1}(j-1)!\binom{n-1}{j-1}\left[(k-1)\stirlingfirst{n-j}{k-1/t}+\stirlingfirst{n-j}{k/t-1}\right].
\end{align}
\end{theorem}
\begin{proof}
Let $j$ be the size of the cycle that contains the element $n$. We choose in $\binom{n-1}{j-1}$ ways the elements join this cycle. If this cycle is coloured by the special colour, the remaining $n-j$ elements build a mixed permutation of the set $\mathcal{MC}(n-j,k,t-1)$. If it is not coloured with the special colour we choose first the colour out of the $(k-1)$ remaining colours, and create a mixed permutation out of the remaining elements in  $\stirlingfirst{n-j}{k-1/t}$ ways.  
\end{proof}
 \begin{theorem}\label{recur2}
Let $n,k$ and $t$ positive integers, with $k,t\leq n$.
\[\stirlingfirst{n}{k/t}=\stirlingfirst{n-1}{k/t-1}+(k-1)\stirlingfirst{n-1}{k-1/t}+(n-1)\stirlingfirst{n-1}{k/t}.\]
\end{theorem}
\begin{proof}
Consider the $n$th element. If it is a singleton, it can be coloured by the special colour, or by any other $(k-1)$ colour. These are 
$\stirlingfirst{n-1}{k/t-1}+(k-1)\stirlingfirst{n-1}{k-1/t}$ possibilities. If it is not a singleton, it can be inserted before any element, and added to the cycle of the element, before it was inserted, which can be done in $(n-1)\stirlingfirst{n-1}{k/t}$ ways.
\end{proof}
\begin{theorem}
For positive integers $n,k,t$ with $k,t\geq n$ we have
\begin{align}\label{recur3}
\stirlingfirst{n}{k/t}=\sum_{j=1}^n\binom{n}{j}(j-1)!\stirlingfirst{n-j}{k-1/t}.
\end{align} 
\end{theorem}
\begin{proof}
Mark one cycle that is not coloured by the special colour. This can be done in $(k-1)\stirlingfirst{n}{k/t}$ ways. Otherwise, we can build a cycle of length $j$ in $\binom{n}{j}(j-1)!$ which we colour in any of the $(k-1)$ non-special colour, and arrange the remaining elements into a mixed coloured permutation in $\stirlingfirst{n-j}{k-1/t}$ ways. After simplication by the factor $(k-1)$ we obtain the theorem. 
\end{proof}
\begin{theorem}
For positive integers $n,k,t$ with $k,t\geq n$ the following statement holds
\begin{align}\label{recur4}
t\stirlingfirst{n}{k/t}=\sum_{j=1}^n\binom{n}{j}(j-1)!\stirlingfirst{n-j}{k/t-1}.
\end{align}
\end{theorem}
\begin{proof}
Mark now one of the cycle coloured by the special colour. The theorem follows by similar argument as \eqref{recur3}. 
\end{proof}
\begin{theorem}
For positive integers $n,k,t$ with $k,t\leq n$ we have
\begin{align*}
\stirlingfirst{n}{k+1/t}=\sum_{j=t+k-1}^n\frac{n!}{j!}\stirlingfirst{j}{t+k-1}(t+k)_k.
\end{align*}
\end{theorem}
\begin{proof}
Let $j+1$ be the greatest element among the least elements of the cycles. Then there are only elements greater than $j$ in the cycle that contains $j+1$, but any of the other cycles contain at least one element from the set $\{1,2,\ldots, j\}$. Construct $t+k-1$ cycles from this set. Now we insert the elements $j+2, j+3, \ldots$ one after the other before any element, adding it to the cycle of the element before it was inserted. There are $(j+1)(j+2)\cdots n$ possibilities to do so. Finally, we colour the cycles such that $t$ cycles obtain the special colours and $k$ cycles obtain distinct colours.
\end{proof}
\begin{theorem}
For positive integers $n,k,t$ with $k,t\leq n$ we have
\begin{align*}
\stirlingfirst{n}{k/t+1}=\sum_{j=t+k-1}^n\frac{n!}{j!}\stirlingfirst{j}{t+k-1}(t+k)_{k-1}.
\end{align*}
\end{theorem}
\begin{proof}
Th proof is similar to the proof of the previous theorem, this time we colour the cycles such that $t+1$ cycles obtain the special colours and $k-1$ distinct colours. 
\end{proof}
Next, we derive the generating functions for the enumerations of the sets $\mathcal{MC}(n,k,t)$ using the symbolic method \cite{Flajolet}. The theory states that the generating function of a set of combinatorial objects can be directly obtained according to a symbolic construction built up of classical basic constructions as sets (SET), sequences (SEQ), cycles (CYC), etc.

Let $\mathcal{A}$ and $\mathcal{B}$ be combinatorial classes with exponential generating functions
$$
A(x)=\sum_{\alpha\in \mathcal{A}}\frac{x^{|\alpha|}}{|\alpha|!}=\sum_{n\geq 0}a_n\frac{x^n}{n!}\mbox{ and }
B(x)=\sum_{\beta\in \mathcal{B}}\frac{x^{|\beta|}}{|\beta|!}=\sum_{n\geq 0}b_n\frac{x^n}{n!}.$$

$a_n$ (resp. $b_n$) is the counting sequence for objects in the class $\mathcal{A}$ (resp.~$\mathcal{B}$) with size $n$. Let $\mathcal{X}$ be the atomic class with generating function $x$. For our purposes we need beyond the sum and product, the following constructions:
\begin{itemize}
\item $\mbox{SEQ}_k(\mathcal{A})$ stands for the class of $k$-sequence is a sequence of length $k$ with parts in $\mathcal{A}$. The translation rule is $(A(x))^k$. $\mbox{SEQ}(\mathcal{A})$ denotes the class of sequences, without taking the length of the sequence into account. The translation rule is $\frac{1}{1-A(x)}$.
\item $\mbox{SET}_k(\mathcal{A})$ denotes the class of $k$-set formed from $\mathcal{A}$, a $k$-sequence module the equivalence relation that two sequences are equivalent when the components of one is the permutation od the components of the other. The corresponding rule is $\frac{(A(x))^k}{k!}$. $\mbox{SET}(\mathcal{A})$ is the class of sets with translation rule $\exp(A(x))$.
\item The notation $\mbox{CYC}_k(\mathcal{A})$ is used for $k$-cycles, set of $k$-sequences modulo the equivalence relation identifying sequences whose elements are cyclic permutations of each other. The rule is $\frac{(A(x))^k}{k}$. $\mbox{CYC}(\mathcal{A})$ is the class of all cycles, and corresponds to $\log\frac{1}{1-A(x)}$.
\end{itemize}
\begin{theorem}
The exponential generating function of the mixed Stirling number of first kind is given by 
\begin{align}\label{genfun}
\sum_{n=0}^{\infty}\stirlingfirst{n}{k/t}\frac{x^{n}}{n!}=\frac{1}{t!}\left(\log\frac{1}{1-x}\right)^{t+k-1}
\end{align} 
\end{theorem}
\begin{proof}
We need to know how we can get a mixed coloured permutation using the basic constructions listed before. A mixed coloured permutation is actually cycles with two extra structures: a set with $t$ elements and an arrangement with $k-1$ elements. We could say, it is a pair of a $t$-set of cycles and of a $(k-1)$-sequence of cycles. Hence, with a little abuse of notation the construction for the mixed coloured permutations is the following:
\begin{align*}
\mathcal{MC}(n,k,t)=\mbox{SET}_t(\mbox{CYC}(\mathcal{X}))\times\mbox{SEQ}_{k-1}(\mbox{CYC}(\mathcal{X})).
\end{align*}
The translation rule gives:
\begin{align*}
\frac{\left(\log\frac{1}{1-x}\right)^{t}}{t!}\times\left(\log\frac{1}{1-x}\right)^{k-1},
\end{align*}
which implies \eqref{genfun} after simplification.
\end{proof}
 \section{$S$-restricted Stirling number of the first kind}
In a series of papers about Stirling numbers \cite{BeRa,BeMeRa1,BeMeRa2} the authors studied the underlying objects, partitions, permutations and lists with the extra condition on the size of the sets, cycles and lists, respectively. For the sake of a comprehensive study of mixed Stirling numbers of the first kind, we follow this idea and derive some results for the number of mixed permutations such that each cycle has length $s$ contained in a given set of integers $S$. These general results include many interesting cases that can be easily obtained by special settings of $S$, as for instance,  $S=\{1,2,\ldots, m\}$ the \emph{restricted mixed Stirling number of the first kind} and $S=\{m, m+1,\ldots\}$ the \emph{associated mixed Stirling number of the first kind}. Furthermore, we can set $S$ the set even numbers, $\mathcal{E}$ or the set of odd numbers, $\mathcal{O}$, etc. Moreover, by an appropriate choice of $S$ our results include results for mixed permutations with forbidden cyclelengths.
\begin{Def}
Given a set $S$ of positive integers and $k,t\leq n$ positive integers. We let $\mathcal{MC}_S(n,k,t)$ denote the set of permutations of $\{1, 2,\ldots, n\}$ into $k$ cycles such that
\begin{itemize}
\item[a.] each cycle has size given in $S$, and
\item[b.] $t$ cycles are coloured with the special colour and the remaining $k-1$ with distinct colours.  
\end{itemize}
We let $\stirlingfirst{n}{k/t}_S$ denote the size of the set $\mathcal{MC}_S(n,k,t)$.
\end{Def}
First, we recall the definition, generating function and a recurrence of the $S$-restricted Stirling numbers of the first kind, on that we heavily rely in this section. For a given set of positive integers $S$, $S$-restricted Stirling numbers, $\stirlingfirst{n}{k}_S$, enumerates the set of permutations of an $n$ element set into $k$ cycles such that each cycle has size contained in $S$. In other words permutations with cycle index $(c_1,c_2,\ldots, c_n)$, such that $c_i\not=0$ if and only if $i\in S$. This number array is a special case of the $(S,r)$-Stirling numbers of the first kind, $\stirlingfirst{n}{k}_{r,S}$, defined in \cite{BeMeRa1}. The generating function and the basic recursion are as follows
\begin{align*}
\sum_{n=k}^{\infty}\stirlingfirst{n}{k}_S\frac{x^n}{n!}=\frac{1}{k!}\left(\sum_{i\geq 1}\frac{x^{k_i}}{k_i}\right)^k.
\end{align*} 
\begin{align*}
\stirlingfirst{n+1}{k}_S=\sum_{s\in S}(s-1)!\binom{n}{s-1}\stirlingfirst{n-s+1}{k-1}_S.
\end{align*}
\begin{theorem}
Given a set of integers $S$, and positive integers $n,k,t$ with $k,t<n$, we have
\begin{align*}
\stirlingfirst{n}{k/t}_S&=(t+k-1)_{k-1}\stirlingfirst{n}{t+k-1}_S,\\
\stirlingfirst{n}{k/t}_S&=\sum_{j=t}^{n-k+1}(k-1)!\binom{n}{j}\stirlingfirst{j}{t}_S\stirlingfirst{n-j}{k-1}_S,\\
\stirlingfirst{n}{k/t}_S&=\sum_{s\in S}\binom{n}{s}(s-1)!\stirlingfirst{n-s}{k-1/t}_S,\\
t\stirlingfirst{n}{k/t}_S&=\sum_{s\in S}(s-1)!\binom{n}{s}(s-1)!\stirlingfirst{n-j}{k/t-1}_S,
\end{align*}
\begin{align*}
\stirlingfirst{n}{k/t}_S=\sum_{s\in S}(s-1)!\binom{n-1}{s-1}\left[(k-1)\stirlingfirst{n-s}{k-1/t}_S+\stirlingfirst{n-s}{k/t-1}_S\right].
\end{align*}
\end{theorem}
\begin{proof}
The identities are straightforward consequences of the theorems \eqref{closed1} \eqref{closed2}, \eqref{recur1}, \eqref{recur3}, and \eqref{recur4}.
\end{proof}
\begin{theorem}
For a given set of positive integers $S$, the exponential generating function for $\stirlingfirst{n}{k/t}_S$ is given by the formula
\begin{align}\label{genfunS}
\sum_{n=0}^{\infty}\stirlingfirst{n}{k/t}_S\frac{x^{n}}{n!}=\frac{1}{t!}\left(\sum_{s\in S}\frac{x^s}{s}\right)^{t+k-1}.
\end{align} 
\end{theorem}
\begin{proof}
The proof is analogue to the proof of \eqref{genfun}. We only need to ensure the condition on the length of the cycles. We use the symbols $\mbox{CYC}_S$ for cyles such that the length cycle is in a given fixed set $S$.
Hence, in this case the construction has to be modified to 
\[\mathcal{MC}_S(n,k,t)=\mbox{SET}_t(\mbox{CYC}_S(\mathcal{X}))\times\mbox{SEQ}_{k-1}(\mbox{CYC}_S(\mathcal{X})),\]
which translates to
\begin{align*}
\frac{1}{t!}\left(\sum_{s\in S}\frac{x^s}{s}\right)^t\times\left(\sum_{s\in S}\frac{x^s}{s}\right)^{k-1}
\end{align*} 
\end{proof}
Next, we derive some formulas that are specific for the $S$-restricted, based on the size of the length of the cycles.
\begin{theorem}
For a given set of positive integers $S$ with $1\in S$, we have 
\begin{align*}
\stirlingfirst{n}{k/t}_S=\sum_{i=0}^t\sum_{j=0}^{k-1}\binom{n}{i,j}(k-1)_j\stirlingfirst{n-i-j}{k-j/t-i}_{S-\{1\}}.
\end{align*}
\end{theorem}
\begin{proof}
Let $i$ be the number of fixed points coloured with the special colour and $j$ the number of fixed points coloured with any other colour. Choose the $i$ elements out of the $n$. Choose $j$ elements out of the $n-i$ elements and colour them with one of the $(k-1)$ colours in $(k-1)(k-2)\cdots(k-j+1)$ ways. The remaining elements $n-i-j$ are ordered into a mixed coloured permutations without fixed points. 
\end{proof}
Let $\stirlingfirst{n}{k/t}_{>1}$ denote the \emph{mixed coloured derangements}, permutations without fixed points. 
\begin{cor}
\begin{align*}
\stirlingfirst{n}{k/t}=\sum_{i=0}^t\sum_{j=0}^{k-1}\binom{n}{i,j}(k-1)_j\stirlingfirst{n-i-j}{k-j/t-i}_{>1}.
\end{align*}
\end{cor}
This theorem can be formulated in a more general form, not only for fixed points, i.e., cycles for length $1$, but for any fixed length $u$.
\begin{theorem}
If $u\in S$, the following identity holds
\begin{align*}
\stirlingfirst{n}{k/t}_S=\sum_{i=0}^t\sum_{j=0}^{k-1}\binom{n}{u(i+j)}\frac{(ui+uj)!}{u^{i+j}i!j!}(k-1)_j\stirlingfirst{n-u(i+j)}{k-j/t-i}_{S-\{u\}}.
\end{align*}
\end{theorem}
\begin{proof}
Let $i$ be the number of cycles of length $u$ coloured by the special colour and $j$ the number of cycles of length $u$ coloured by any other colour. We choose $u(i+j)$ elements and order them into cycles of length $u$ according to the following pocedure: we arrange the elements, and take the first $u$ elements as a cycle, the next $u$ elements and so on. Clearly, some double counts occur, any of the elements in a cycle could be the starting element, and the order of the cycles are not important, only that the first $i$ will be coloured by the special colour. These are
$\binom{n}{u(i+j)}\frac{(u(i+j))!}{u^{i+j}i!j!}$ possibilities to do so. 
We need to choose the colours for the non-special coloured cycles in $(k-1)_j$ ways. The remaining $n-u(i+j)$ elements build a mixed permutation without any cycle of length $u$.
\end{proof}
\section{Mixed r-Stirling numbers of the first kind}
$r$-Stirling numbers received a lot of attention recently in the research. $r$-Stirling numbers of the first kind counts the number of permutations that can be decomposed into exactly $k$ cycles such that the first $r$ elements $\{1,2,\ldots, r\}$ are in distinct cycles. In this section we introduce $r$-mixed Stirling numbers of the first kind following the original concept. Before, we derive two formulas for $r$-Stirling numbers of the first kind, $\stirlingfirst{n}{k}_r$. Our formulas contain also the Lah-numbers, $\lah{n}{k}$, which counts the number of partitions of the an $n$ element set into  $k$ linearly ordered subset and is given by the closed formula
\[\lah{n}{k}=\binom{n-1}{k-1}\frac{n!}{k!}.\]
\begin{theorem}\label{rsf}
Let $n,k$ be positive integers. The $r$-Stirling numbers of the first kind given by
\[\stirlingfirst{n}{k}_r=\sum_{\ell=0}^{n-k+1}r!\binom{n-r}{\ell}\lah{\ell}{r}\stirlingfirst{n-r-\ell}{k-r}.\]
\end{theorem}
\begin{proof}
We put first the $r$ elements into distinct blocks. Next, we choose $\ell$ elements out of the remaining non-special elements $n-r$ and partition them into $r$ linearly ordered sets and merge these blocks with one of the $r$ blocks containing already one special element, simply by writing the sequence after the special element. The remaining $n-r-\ell$ elements build a permutation into $k-r$ cycles. 
\end{proof}
The next theorem is a symmetric version of the above theorem.
\begin{theorem}
\begin{align*}
\stirlingfirst{n}{k}_r=\sum_{\ell=k-r}^{n-r}r!\binom{n-r}{\ell}\stirlingfirst{\ell}{k-r}\lah{n-r-\ell}{r}.
\end{align*}
\end{theorem}
\begin{proof}
Choose now $\ell$ elements for cycles that do not contain any of the special elements $\{1,2,\ldots, r\}$.
\end{proof}
\begin{Def}
We call a mixed coloured permutation \emph{$r$-mixed coloured permutation} if the elements $\{1,2,\ldots, r\}$ are in distinct cycles. 
We let $\mathcal{MC}_r(n,k,t)$ denote the set of all $r$-mixed $(t,1,\ldots, 1)$-coloured permutations of the set $\{1,2,\ldots, n\}$ coloured by exactly $k$ colours. We let $\stirlingfirst{n}{k/t}_r$ denote the size of the set $\mathcal{MC}_r(n,k,t)$.
\end{Def}
\begin{theorem}
Let $n,k,t$ and $r$ be a positive integers. The number of $r$-mixed coloured permutations is 
\begin{align*}\stirlingfirst{n}{k/t}_r=(t+k-1)_{k-1}\stirlingfirst{n}{t+k-1}_r\end{align*}
\end{theorem}
\begin{proof}
First, order the $n$ elements into $t+k-1$ cycles such that the elements $\{1,2,\ldots, r\}$ are in distinct cycles. Next, colour the cycles to obtain a mixed coloured permutations, choose $k-1$ cycles and for $k-1$ colours, and order these colours to the chosen cycles. 
\end{proof}
\begin{theorem}
Let $n,k,t$ and $r$ be a positive integers. The $r$-mixed Stirling numbers of the first kind is given by
\[\stirlingfirst{n}{k/t}_r=\sum_{\ell=k+t-r-1}^{n-r}\sum_{i=0}^{\min(t,r)}r!\binom{r}{i}\binom{k-1}{r-i}\binom{n-r}{\ell}\lah{n-r-\ell}{r}
\stirlingfirst{\ell}{k-r+i/t-i}.\]
\end{theorem}
\begin{proof}
First, put the $r$ elements $\{1,2,\ldots, r\}$ into distinct cycles and colour these cycles. Let $i$ be the number of cycles that contain one of the $r$ elements from the set $\{1,2,\ldots, r\}$ and are coloured by the special colour ($0\leq i\leq \min(t,r)$). Then there are $\sum_i\binom{r}{i}\binom{k-1}{r-i}(r-i)!$ possibilities to do so, choosing the $i$ elements for the special colour and choosing the colours and order the colours for the remaining $r-i$ cycles. Next, we a mixed coloured permutation without any elements of the set $\{1,2,\ldots,r\}$. Choose $\ell$ elements for this out of the set $\{r+1,r+2,\ldots, n\}$ in $\stirlingfirst{\ell}{k-(r-i)/t-i}$ ways. Finally, ,,fill'' the $r$ cycles containing elements $\{1,2,\ldots, r\}$ with the remaining $n-r-\ell$ elements by partitioning them into $r$ linearly ordered sets and assigning to each list a number between $1$ and $r$.    
\end{proof}
\section{Partial Bell-polynomials and mixed permutations}
Bell-polynomials, $B_{n,k}(x_1,x_2,\ldots)=B_{n,k}(x_j)_{j\geq 1}$ were introduced by Bell, as a mathematical tool for representing the $n-$th derivative of composite functions and were intensively studied since then by many authors \cite{Comtet,Mihoubi1}. 
The exponential partial Bell polynomial is defined by the generating function
\begin{align}\label{Bell}
\sum_{n\geq k}B_{n,k}(x_j)\frac{t^n}{n!}=\frac{1}{k!}\left(\sum_{m\geq 1}x_m\frac{t^m}{m!}\right)^k,
\end{align}
and is given explicitly by
\begin{align}
B_{n,k}(x_1,x_2,\ldots)=\sum_{p(n,k)}\frac{n!}{k_1!\cdots k_n!}\left(\frac{x_1}{1!}\right)^{k_1}\left(\frac{x_2}{2!}\right)^{k_2}\cdots\left(\frac{x_n}{n!}\right)^{k_n},
\end{align}
where the sum goes over all vectors $(k_1,\ldots, k_n)$ with $\sum_{i\geq 1}{k_i}=k$ and $\sum_{i\geq 1}{ik_i}=n$.
For instance,
$B_{6,2}(x_2,x_3,x_4)=15x_2x_4+10x_3$, since  $[6]$ can be partitioned into two blocks with size $2/4$ and $3/3$ in $15$ and $10$ ways, respectively. 
It is well-known \cite{Comtet} that for particular settings of the variables $x_j$ the exponential partial Bell polynomials reduce to the following sequences:
\begin{align*}
B_{n,k}(1,1,1,\ldots)&=\stirling2{n}{k}\\
B_{n,k}(0!,1!,2!,\ldots)&=\stirlingfirst{n}{k}\\
B_{n,k}(1!,2!,3!,\ldots)&=\lah{n}{k}.
\end{align*}
Mihoubi \cite{Mihoubi1} presented the combinatorial interpretation as follows. For given $(a_j)_{j\geq 1}$, $B_{n,k}(a_j)$ counts the number of partitions of an $n$ element set into $k$ blocks such that the blocks of size $j$ can be colored by $a_j$ colors, $B_{n,k}((j-1)!a_j)$ counts the number of permutations of an $n$ element set into $k$ cycles such that the cycles of size $j$ can be colored by $a_j$ colors, and $B_{n,k}(j!a_j)$ counts the number of partitions of an $n$ element set into $k$ lists such that each list can be colored by $a_j$ colors. 
In the same paper \cite{Mihoubi1} the authors define the $r$-partial Bell polynomials, $B_{n,k}^{(r)}(x_i,y_i)$.
\begin{align}
B_{n,k}^{(r)}(x_i,y_i)=\sum_{\Lambda(n,k,r)}\left[\frac{n!}{k_1!k_2!\cdots}\left(\frac{x_1}{1!}\right)^{k_1}\left(\frac{x_2}{2!}\right)^{k_2}
\cdots\right]\left[\frac{r!}{r_0!r_1!\cdots}\left(\frac{y_1}{1!}\right)^{k_1}\left(\frac{y_2}{2!}\right)^{k_2}\cdots\right],
\end{align}
where $\Lambda(n,k,r)$ is the set of all nonnegative integer sequences $(k_i)_{i\geq 1}$ and $(r_i)_{i\geq 0}$ such that $\sum_{i\geq 1}k_i=k$, $\sum_{i\geq 0}{r_i}=r$ and $\sum_{i\geq 1}i(k_i+r_i)=n$. $B_{n,k}^{(0)}(x_i,y_i)=B_{n,k}(x_i)$ and 
\begin{align*}
B_{n,k}^{(r)}(1,1,1\ldots)&=\stirling2{n}{k}_r\quad \mbox{$r$-Stirling numbers of the second kind \cite{Broder}},\\
B_{n,k}^{(r)}(0!,1!,2!,\ldots)&=\stirlingfirst{n}{k}_r\quad \mbox{$r$-Stirling numbers of the first kind \cite{Broder}},\\
B_{n,k}^{(r)}(1!,2!,3!,\ldots)&=\lah{n}{k}_r\quad \mbox{$r$-Lah numbers \cite{Nyul}},\\
B_{n,k}^{(r)}(1,m,m^2,\ldots;1,1,\ldots)&=W_{m,r}(n,k)\quad \mbox{$r$-Whitney numbers \cite{Cheon}}.
\end{align*}
For our purposes we modify here slightly this definition. 
\begin{Def}
\begin{align*}
B_{n,k,t}^*(x_{\ell};y_{\ell})=\sum_{p(n,k,t)}\frac{n!}{t_1!t_2!\cdots t_n!}\left(\frac{x_1}{1!}\right)^{t_1}\left(\frac{x_2}{2!}\right)^{t_2}\cdots
\left(\frac{x_n}{n!}\right)^{t_n}\frac{(k-1)!}{k_1!k_2!\cdots k_n!}\left(\frac{y_1}{1!}\right)^{k_1!}\left(\frac{y_2}{2!}\right)^{k_2}\cdots\left(\frac{y_n}{n!}\right)^{k_n},.
\end{align*}
where the sum runs over all sequences $p(n,k)=\{(t_i)_{i\geq 1}; (k_i)_{i\geq 1}:k_i, t_i\in \N, \sum_{i\geq 1}t_i=t, \sum_{i\geq 1}k_i=k-1, \sum_{i\geq 1}i(t_i+k_i)=n\}$. 
\end{Def}
Combinatorially seen, $B_{n,k,t}^*(x_1,x_2,\ldots; y_1,y_2,\ldots)$ counts the number of partitions of $n$ into $t$ blocks such that the blocks of length $i$ can be coloured with $x_i$ colours and $k-1$ ordered blocks such that a block of length $i$ can be coloured with $y_i$ colours. This follows from the well-known fact that the number of partitions of $[n]$ into $k$ blocks, $k_i$ denotes the number of blocks with cardinality $i$ is:
\begin{align*}\sum_{\pi(n,k)}\frac{n!}{k_1!(1!)^{k_1}k_2!(2!)^{k_2}\cdots k_n!(n!)^{k_n}}.\end{align*} 
\begin{Def}
Let $(a_n;n\geq 1)$ be a sequence of nonnegative integers. The number $B_{n,k,t}^{*}(a_j)$ ($x_i=y_i=a_i$) counts the number of mixed partitions of an $n$ element set into $k$ blocks such that $t$ blocks are labeled by $1$ and the others by $2,\ldots, k$ and such that any block can be coloured with $a_j$ colors.  
\end{Def} 
Then the followings  hold
\begin{align}\label{Bell1}
B_{n,k,t}^{*}(a_1,a_2,\ldots)&=\frac{1}{t!}\sum_{\substack{\sum_{i=1}^{n+k-t} n_i=n}}\binom{n}{n_1,n_2,\ldots,n_{k+t-1}}a_{n_1}a_{n_2}\cdots a_{n_{k+t-1}}
\end{align}
\begin{align}\label{Bell2}
B_{n,k,t}^*(a_j)=\frac{n!}{t!}\sum_{\substack{\sum_{i=1}^{k+r-1} n_i=n}}\frac{a_{n_1}a_{n_2}\cdots a_{n_{k+t-1}}}{n_1!n_2!\cdots n_{k+t-1}!}
\end{align}
\begin{align}\label{Bell3}
\sum_{n=k}^{\infty}B_{n,k,t}^{*}(a_j)\frac{z^n}{n!}=\frac{1}{t!}\left(\sum_{m\geq 1}a_m\frac{z^m}{m!}\right)^{t+k-1}.
\end{align}
\begin{proof}
There are $\frac{1}{t!}\binom{n}{n_1,n_2,\ldots, n_t}a_{n_1}a_{n_2}\cdots a_{n_t}$ possibilities to partition elements out of the set $\{1,2,\ldots, n\}$ into $t$ blocks of size $n_1$, $n_2$,$\ldots$, $n_t$, and  colour them by $a_{n_1}$, $a_{n_2}$, $\ldots$, $a_{n_t}$ colours. The remaining elements $n-(n_1+n_2\ldots +n_t)$ can be arranged into an ordered partition with blocks of size $n_{t+1}$, $n_{t+2}$, $\ldots$, $n_{t+k-1}$ and coloured with $a_{n_{t+1}}$, $a_{n_{t+2}}$, $a_{n_{t+k-1}}$ colours in  $\binom{n-(n_1+n_2\ldots+n_t)}{n_{t+1},n_{t+2},\ldots, n_{t+k-1}}$ ways. This implies \eqref{Bell1}. We get \eqref{Bell2} from \eqref{Bell1} by simplification. 
\begin{align*}
\sum_{n\geq \max{k,t}}B^*_{n,k,t}(a_1,a_2,\ldots)\frac{z^n}{n!}\\
&=\sum_{n\geq \max{k,t}}\frac{1}{t!}\sum_{n_1+\cdots +n_{t+k-1}=n}\frac{a_1a_2\cdots a_{t+k-1}}{n_1!n_2!\cdots n_{t+k-1!}}z^n\\
&=\frac{1}{t!}\sum_{n_i\geq 1}\frac{a_{n_1}\cdots a_{n_{t+k-1}}}{n_1!\cdots n_{t+k-1}!}z^{n_1+\cdots n_{t+k-1}}\\
&=\frac{1}{t!}\left(\sum_{j\geq 1}a_j\frac{z^j}{j!}\right)^{t+k-1},
\end{align*}
and \eqref{Bell3} follows.
\end{proof}
We have
\begin{align*}
\stirlingfirst{n}{k/t}=B^*_{n,k,t}(0!,1!,2!,\ldots)
\end{align*}
Let $c_S$ be the \emph{characteristic sequence} of a given set $S$ of positive integers, i.~e.~ $c_S=\{c_i|i=1,2,\ldots, c_i=1, \mbox{if} i\in S\mbox{ and }c_i=0,\mbox{ if } i\not\in S \}$.
Then we have
\begin{align*}
\stirlingfirst{n}{k/t}_S=B^*_{n,k,t}((i-1)!c_S).
\end{align*}
For the general expression $B^*_{n,k,t}((i-1)!a_i)$ we have the following combinatorial interpretation. $B^*_{n,k,t}((i-1)!a_i)$ is the number of permutations of $[n]$ into $k+t-1$ double labeled cycles: the first label of $t$ cycles is $1$, while the first labels of the other cycles are different out of the set $\{2,3,\ldots, k\}$. The second label depends on the size of the cycle, each cycle of length $i$ receives a label out of the set $\{1,2,\ldots, a_i\}$. 
\begin{remark}
$B^*_{n,k,t}(1,1,\ldots)$,respectively $B^*_{n,k,t}(c_S)$, counts the number of mixed partitions defined and studied in \cite{BeYa,Yaq}, respectively the mixed partitions with block sizes contained in a given set of positive integers $S$. 
\end{remark}


\begin{thebibliography}{40}
\bibitem{BeYa} S.~Barati, B.~Bényi, A.~Jafarzaded, and D.~Yaqubi, Mixed restricted Stirling numbers,\textit{Acta Math.~Hungar.~} (2019) https://doi.org/10.1007/s10474-019-00915-8
\bibitem{BeRa} B.~Bényi and J.~L.~Ramírez, Some applications of $S$-restricted set partitions, \textit{Period.~Math.~Hung.~} \textbf{78(1)} (2019), 110--127.
\bibitem{BeMeRa1} B.~Bényi, M.~Méndez, J.~L.~Ramírez and T.~Wakhare, Restricted r-Stirling numbers and their combinatorial applications \textit{Appl.~Math.~Comput.~}, \textbf{348} (2019), 186--205.
\bibitem{BeMeRa2} B.~Bényi, M.~Méndez, and J.~L.~Ramírez, Generelaized ordered set partitions, preprint
\bibitem{Broder} A.~Z.~Broder, The $r$-Stirling numbers, \textit{Discrete Math.~}, \textbf{49} (1984), 241-259.
\bibitem{Cheon} G.~S.~Cheon and J.~H.~Jung, $r$-Whitney numbers of Dowling lattices, \textit{Discrete Math.~}, \textbf{312} (2012) 2337--2348.
\bibitem{Comtet}
L.~Comtet,
\textit{Advanced Combinatorics}, 
D. {R}eidel {P}ublishing {C}o. ({D}ordrecht, {H}olland), 1974.
\bibitem{Flajolet}
P.~Flajolet and R.~Sedgewick, \emph{Analytic Combinatorics}, Cambridge University Press, Cambridge, 2009.
\bibitem{Mihoubi1}
M.~Mihoubi and M.~Rahmani, The partial r-Bell polynomials,\textit{Afrika Mat.~} \textbf{28(7-8)} (2017), 1167--1183.
\bibitem{Nyul} G.~Nyul and G.~Rácz, The $r$-Lah numbers, \textit{Discrete Math.~} \textbf{338} (2015), 1660--1666.
\bibitem{OEIS}  N.~J.~A.~Sloane, The On-Line Encyclopedia of Integer Sequences, URL http://oeis.org
\bibitem{Yaq} D.~Yaqubi, M.~Mirzavaziri, and Y.~Saeednezhad, Mixed $r$-stirling number of the second kind, \textit{Online J.~ Anal.~ Comb.~}, \textbf{11} (2016), 5.
\end{thebibliography}
\end{document}